\documentclass[12pt, a4paper,  reqno, english]{amsart}
\usepackage[all]{xy}
\usepackage[margin=0.9in]{geometry} 
\usepackage{amssymb,amsmath,amsthm}
\numberwithin{equation}{section}
\usepackage{lineno}

\newcommand{\Z}{\mathbb{Z}}

\newtheorem{proposition}{Proposition}
\newtheorem{lemma}{Lemma}
\newtheorem{corollary}{Corollary}
\newtheorem{conj}{Conjecture}

\newtheorem{example}{Example}

\theoremstyle{definition}

\renewcommand{\mod}[1]{\hspace{-2.9mm}\pmod{#1}}

\newcommand{\ben}{\begin{enumerate}}
\newcommand{\een}{\end{enumerate}}
\newcommand{\eit}{\begin{itemize}}
\newcommand{\beq}{\begin{equation}}
\newcommand{\eeq}{\end{equation}}

\renewcommand{\leq}{\leqslant}
\renewcommand{\geq}{\geqslant}

\usepackage{color}
\definecolor{red}{rgb}{1,0,0}
\definecolor{blue}{rgb}{.2,.6,.75}
\definecolor{green}{rgb}{.4,.7,.4}

\begin{document}

\title{Two dimensional covering systems and possible prime producing $a^m-b^n$} 
\author{Andrew Granville and Francesco Pappalardi}
\thanks{Many thanks to Keqin Liu as well as the referee for several helpful remarks and observations.}
 
\begin{abstract}  We exhibit a new application of two dimensional covering systems,  examples of integer pairs $(a,b)$ for which $a^m-b^n$ has a prime divisor from a given finite set of primes, for every pair of integers $m,n\geq 0$.    This leads us to conjecture what are \emph{the only possible} obstructions  to $|a^m-b^n|$ taking on infinitely many distinct prime values.
\end{abstract}
\maketitle

\section{Introduction}
 
 We begin by noting that 
 \[
 41^m-34^n
 \] 
 is divisible by 3,5 or 7  (that is $(41^m-34^n,3\cdot 5\cdot 7)>1$) for all integers $m,n\geq 0$:
 \[
 41^m-34^n\equiv
 \begin{cases}
 (-1)^m-1 \qquad \ \pmod 3 &\text{ so divisible by 3 if } m\equiv 0 \pmod 2;\\
\qquad \  1- (-1)^n \pmod 5 &\text{ so divisible by 5 if } n\equiv 0 \pmod 2;\\
   (-1)^m-(-1)^n \pmod 7 &\text{ so divisible by 7 if } m\equiv n\pmod 2.
 \end{cases}
 \]
Therefore $|41^m-34^n|$ is either composite or it equals  3,5 or 7. 
\footnote{There is another way to  establish that $ 41^m-34^n$ is composite when $m\equiv n\equiv 0\pmod 2$: we can write
 $m=2M, n=2N$ so that $41^m-34^n=(41^M-34^N)(41^M+34^N)$. }

 Siegel's $S$-unit theorem \cite{Sg2} states that there are only finitely many solutions to $a+b=c$ in coprime positive integers $a,b,c$ whose prime factors all come from a finite set $S$. Taking $S=\{ 2,3,5,7,17,41\}$ we deduce that there are only finitely many pairs $(m,n)$ for which $|41^m-34^n|$ is a power of $3,5$ or $7$. In particular $|41^m-34^n|$ can be prime for only finitely many pairs of integers $(m,n)$.\footnote{But $|41^m-34^n|$ can take prime values,   like $41^1-34^1=7$.}
 The same argument shows that for any given integers $a$ and $b$, if there exists an integer $Q$ such that 
\begin{equation} \label{eq: gcdQ}
(a^m-b^n , Q)>1 \text{ for all  positive integers } m,n ,
\end{equation}
then $|a^m-b^n|$ can be a prime or a prime power for only finitely many pairs of integers $(m,n)$.\footnote{One might want to make a stronger conjecture; for example, that if  \eqref{eq: gcdQ}  holds then $|a^m-b^n|$ can be a prime or prime power for no more than two pairs of integers $(m,n)$.}

Perhaps this is the only obstruction to there being prime values of $a^m-b^n$?
If so this leads us to make the following conjecture: 

\begin{conj}
For any given integers $a,b \geq 2$,  such that neither  $a$ nor $b$ is the power of an  integer,\footnote{ That is, there does not exist integers $A$ and $k\geq 2$ for which $a=A^k$, nor  integers $B$ and $\ell\geq 2$ for which $b=B^\ell$. We revisit this situation in Section \ref{sec: powers2}.}
 There are infinitely many primes of the form
 \[
 |a^m-b^n| \text{ where } m,n \text{ are positive integers},
 \]
 unless there exists a non-zero integer $Q$ for which \eqref{eq: gcdQ} holds.
 \end{conj}

We will show how to efficiently construct \emph{all}   triples $(a,b,Q)$ satisfying \eqref{eq: gcdQ}.\footnote{Though we do not have an efficient way to determine whether such a $Q$ exists for given integers $a$ and $b$.}
If there is no  obstruction as in \eqref{eq: gcdQ} then
 we conjecture that there is a constant $c_{a,b}>0$ such that 
 \begin{equation} \label{eq: PrimeCount}
  \#\{ |a^m-b^n|\leq x: |a^m-b^n| \text{ is prime}\} \sim  c_{a,b} \log x,
 \end{equation}
which we back with some computational evidence.  The constant $c_{a,b}$ is defined in the text as a limit of a sequence of positive constants; we cannot prove that the limit exists, nor that if the limit exists then it is non-zero, but we believe that the limit does exist and that it is non-zero.
 
 We also believe that the set of pairs of integers $(a,b)$ for which there is  an integer $Q$ satisfying \eqref{eq: gcdQ}  has a positive density strictly smaller than $1$.

 \subsection{A mixed case} \label{sec: powers2}
  The case $a=51, b=64$ (so that $b$ is a square) is a little different: Here
 \begin{align*}
 51^m-64^n&\equiv 
 \begin{cases}
  \qquad \ 1- (-1)^n \pmod 5 &\text{ so is divisible by 5 \ if }   n\equiv 0\ \pmod 2;\\
   (-1)^m-(-1)^n \pmod {13} &\text{ so is divisible by 13 if } m\equiv n\pmod 2;
 \end{cases} \\
 &=51^{m} - 8^{2n} \text{  is divisible by } 51^{m/2} + 8^n \text{ if } m\equiv 0\pmod 2.
 \end{align*}
Arguing as before with the $S$-unit theorem,\footnote{The $S$-unit theorem also implies that $51^{m/2} - 8^n=\pm 1$ for only finitely many pairs of integers $(m,n)$.}
 we deduce that $|51^m-64^n|$ has at least two distinct prime factors for all but finitely many pairs of positive integers $(m,n)$.

 We can generalize conjecture 1 in the cases that $a$ or $b$ is a  power:

\begin{conj}
For any given integers $a,b \geq 2$,  select $k$ and $\ell$  maximal
with $a=A^k, b=B^\ell$. There are infinitely many primes of the form
 \[
 |a^m-b^n| \text{ where } m,n \text{ are positive integers},
 \]
 unless there exists a non-zero integer $Q$ for which 
 \begin{equation} \label{eq: gcdQ2}
 (a^m-b^n , Q)>1 \text{ for all  positive integers } m,n \text{ with } (m,\ell)=(n,k)=1.
  \end{equation}
\end{conj}

 If \eqref{eq: gcdQ2} holds then we can deduce that there are only finitely many primes of the form $|a^m-b^n|$.


\subsection{Remarks}  \label{sec: Rmks}
(i)\ Since $(a,b)$ divides $a^m-b^n$ for all integers $m,n\ge1$,  we can take $Q=(a,b)$ if $(a,b)>1$. 
We therefore assume henceforth that $(a,b)=1$.  This implies that $(ab,a^m-b^n)=1$ so we may assume, without loss of generality that 
$(Q,ab)=1$.

(ii)  Now  $a^m-b^n\equiv 1^m-1^n=0 \pmod {(a-1,b-1)}$ so we can take $Q=(a-1,b-1)$ if $(a-1,b-1)>1$.
 
 (iii)\ Extending the question to $m,n\ge0$ would add the sequences $a^m-1$ and $b^n-1$. The elements of these sequences are divisible by $a-1$ and $b-1$ respectively and therefore composite for $m,n\ge2$, respectively, unless $a=2$ or $b=2$. However primes of the form $2^m-1$ are Mersenne primes, a well-studied topic so we ignore it.
 
 (iv)\ We have that $(a^m-b^n , Q)>1$ for all  $m,n\geq 1$ if and only if $(a^m-b^n , Q)>1$ for all coprime $m,n\geq 1$. To see this note that if 
 $(m,n)=g$ then write $m=Mg,n=Ng$ so $(M,N)=1$. Then $a^M-b^N$ divides $(a^M)^g-(b^N)^g=a^m-b^n$ so that 
  $(a^M-b^N,Q)$ divides $(a^m-b^n , Q)$, and therefore $(a^m-b^n , Q)\geq (a^M-b^N,Q)>1$.
  
  (v) Suppose that $b=a\pm 1$. If $m\equiv n\equiv 1 \pmod {\phi(Q)}$ for any given integer $Q$, then  $a^m-b^n\equiv a-b=\mp 1 \pmod Q$ and therefore $(a^m-b^n,Q)=1$. This implies that there can be no integer $Q$ satisfying \eqref{eq: gcdQ}.
 
   (vi) How about $a^m+b^n$?  If $a+b$ is even then $a^m+b^n$ is even, and so composite,  for all $m,n\geq 1$.
If $a+b$ is odd then Keqin Liu (in email corrrespondence) notes that  if  $m\equiv n\equiv 0 \pmod {\phi(Q)}$ for any given integer $Q$, then  $a^m+b^n\equiv 1+1\equiv 2 \pmod Q$ as $(Q,ab)=1$ and therefore $(a^m+b^n,Q)=(a^m+b^n,2,Q)=1$. This implies that there can be no integer $Q$ for which  $(a^m+b^n,Q)>1$ for all $m,n\geq 1$.

 \section{Two-dimensional covering systems}

 A set of integer triples $\{ (u_i,v_i,r_i): i=1,\dots,k\}$ is a \emph{two-dimensional covering system} if for every pair of integers $(m,n)$ there exists $i$ for which
 \[
 mv_i\equiv nu_i \pmod {r_i}.
 \]
 For example, $\{(1,0,2), (0,1,2), (1,1,2)\}$ is a two-dimensional covering system.
 
  In this section we show that every   triple of positive integers $\{a,b,Q\}$, where $Q$ is squarefree and $(a,b)=1$ for which
$ (a^m-b^n , Q)>1$ for all $m,n\geq 1$, can be obtained from a two-dimensional covering system. And, vice-versa, given a two-dimensional covering system we can find all of the corresponding triples $\{a,b,Q\}$.

\subsection{Two-dimensional congruences}
For given integers $u,v,r$ with $r\geq 1$ we define
\[
S(u,v,r)=\{ (m,n)\in \Z^2: mv\equiv nu \pmod r\}.
\]
We call $(u,v,r)$ and $(U,V,R)$ \emph{equivalent} if $S(u,v,r)=S(U,V,R)$.
We call $(u,v,r)$ \emph{semi-reduced} if $u$ divides $r$, $(v,u)=1$   {and $1\leq v\leq r$.
If $(u,v,r)$ is semi-reduced then $S(u,v,r)$ is precisely the sublattice of $\Z^2$ generated by the two vectors $(0,r/u)$ and $(u,v)$.
(To prove this we observe that $u$ must divide $m$ so writing $m=uk$ we obtain $n\equiv vk\pmod {r/u}$.)
For example, $S(5,1,15)=\langle  (0,3), (5,1) \rangle =\langle  (0,3), (5,4) \rangle =S(5,4,15)$, so semi-reduced sets are not necessarily distinct.

In the next lemma we show that every $(u,v,r)$ is equivalent to a semi-reduced triple; and that
 reduced semi-reduced triples form finite equivalence classes, but not necessarily of size one.





}


\begin{lemma}  \label{lem: 1} (a) Every triple $(u,v,r)$ with $r\geq 1$  is equivalent to a semi-reduced triple. \\
(b)   If  $(u,v,r)$ and $(U,V,R)$ are semi-reduced then $S(u,v,r)=S(U,V,R)$ if and only if $u=U, r=R$ and $V\equiv v  \pmod {r/u}$ with $(V,u)=1$.
  \end{lemma}
 
 \begin{proof}  (a):  We begin by replacing $u$ and $v$ by their least positive residues $\pmod r$.

 We may assume that gcd$(u,v,r)=1$ for if $g=$gcd$(u,v,r)$ and $u=gU, v=gV, r=gR$ then $mv\equiv nu \pmod r$ if and only if $mV\equiv nU \pmod R$ and so $S(u,v,r)=S(U,V,R)$.

 We now show that we may assume that $u$ divides $r$: If not, let
  $U=(u,r)$ and write $u=Uh, r=UR$ with $(h,R)=1$. Now $(U,v)=(u,v,r)=1$ and so if $mv\equiv nu \pmod r$ then $U$ divides $mv$ and so $U$ divides $m$. Writing $m=MU$, we have $mv\equiv nu \pmod r$ if and only if 
 $Mv\equiv n h \pmod { R}$  which holds  if and only if  $n \equiv M vk \pmod { R}$ where
$k$ is the inverse of $h \pmod R$. 
Select $V$ to be a positive integer $\leq   RU=r$ such that $V\equiv vk \pmod R$ and $(V,U)=1$. Therefore $n \equiv M vk \pmod { R}$ if and only if $n \equiv MV \pmod { R}$ and then, multiplying through by $U,$ this holds if and only if $nU \equiv m V \pmod r$.  Thus $S(u,v,r)=S(U,V,r)$ and  $ U=(u,r)$ divides $r$ with $(U,V)=1$ and  $1\leq V\leq r$.  

We now deduce that $(u,v)=(u,v,r)=1$ as $(u,r)=u$.  Moreover given any  triple $(u,v,r)$ for which $u$ divides $r$ and  $(u,v)=1$, we replace 
$v$ by $V$ the smallest positive residue of $v \pmod r$ and we obtain a semi-reduced triple $(u,V,r)$ with $S(u,v,r)=S(u,V,r)$.
\medskip

(b): We are assuming that $mv\equiv nu \pmod r$ if and only if $mV\equiv nU \pmod R$.
Now $u$ divides $r$ so $u$ divides $mv$ and therefore $m$ as $(u,v)=1$. Therefore $(m,n)=(u,v)$ gives the smallest $m$-value in a solution, and so $u=U$. More generally if $m=u$ then we see that the arithmetic progressions $v  \pmod {r/u}$ and $V \pmod {R/u}$
contain the same integers so $r/u=R/U$ and thus $r=R$ and $V\equiv v \pmod {r/u}$.
   \end{proof}

 The covering system $\{(1,2,2), (2,1,2), (1,1,2)\}$ is made out of semi-reduced triples and is equivalent to the one in the above example.
 
 A semi-reduced triple $(u,v,r)$ is said to be \emph{reduced} if $v$ is the least positive integer among those integers $w$ such that $(u,w,r)$ is semi-reduced and equivalent to $(u,v,r)$.
 

 \begin{proposition} \label{Prop from Lemma 1} Every triple $(u,v,r)$ with $r\geq 1$ and $u$ and $v$ not both $0 \pmod r$, is equivalent to a unique reduced triple.
 \end{proposition}  

 \begin{proof}  
By Lemma \ref{lem: 1}(a) $(u,v,r)$ is equivalent to a semi-reduced triple, and by Lemma \ref{lem: 1}(b) this semi-reduced triple is equivalent to a unique reduced triple.
 \end{proof}

\begin{proposition} \label{Prop Generators} Let $p$ be a prime that does not divide the integers $a$ and $b$. Let $r=\text{ord}_p(a), s=\text{ord}_p(b), L=[r,s]$  and  $u=\frac r{(r,s)}=\frac Ls$.\\
(a) There exists a residue $g\pmod p$ with $\text{ord}_p(g)=L$ and an integer $v$
such that $a\equiv g^{v} \pmod p$ and $b\equiv g^{u} \pmod p$ where $(u,v,L)$ is a semi-reduced triple.\\
(b) There exists a unique residue $g\pmod p$ with $\text{ord}_p(g)=L$  and a unique integer $v$
such that $a\equiv g^{v} \pmod p$ and $b\equiv g^{u} \pmod p$ where $(u,v,L)$ is a reduced triple.\\
We denote these derived values as $g_p(a,b), u_p(a,b), v_p(a,b), L_p(a,b)$.
\end{proposition} 

\begin{proof}  (a):  We begin by observing that there exists a residue $g \pmod p$ of order $L$ for which $b\equiv g^u\pmod p$:   Since the residues $\pmod p$ form a cyclic group of order $p-1$ we know that $r$ and $s$, being orders of residues, must divide $p-1$, and so $L$ divides $p-1$.
But then the residues of order dividing $L$ form a cyclic subgroup generated say by $h$ which is of order $L$. Therefore   there exists an integer $k$ for which $b\equiv h^k \pmod p$.  But $b$ has order $s$ and so 
$(k,L)=\frac{L}s$ which implies there exists an integer $i$, coprime with $s$, for which $k=\frac{L}s i$.  Next we select an integer $I$ which is $\equiv i \pmod s$ and coprime with $L$ which is easily done using the Chinese Remainder Theorem, so that
$k\equiv \frac{L}s I \pmod {L}$. Now let $g\equiv h^I \pmod p$ and then 
$b\equiv h^k \equiv h^{ \frac{L}s I}\equiv g^{\frac{L}s} \pmod p$ as claimed. Moreover $g$ has order $L$ as $(I,L)=1$.

There exists an integer $v$ for which $a\equiv g^v$ and $(v,L)=\frac{L}r$ as $a$ has order $r \pmod p$. Therefore $v=\frac{L}r j$ where $(j,r)=1$, and we select $1\leq j\leq r$.  We deduce that $(u,v)=(\frac{L}rj, \frac{L}s) =( j\frac s{(r,s)},\frac r{(r,s)}) =1  $ as $(j,r)=1$,
and   $0<v= \frac{L}r j\leq L$ as $0<j\leq r$, and therefore  $(u,v,L)$ is semi-reduced.

(b):\ Suppose that in (a) we have  $h\pmod p$ with $\text{ord}_p(h)=L$ and an integer $w$
such that $a\equiv h^w  \pmod p$ and $b\equiv h^{u} \pmod p$ where $(u,w,L)$ is a semi-reduced triple.

 By   Lemma \ref{lem: 1}(b) we know that $(u,w,L)$  is equivalent to a unique reduced triple $(u,v,L)$ where
$v\equiv w \pmod{L/u}$ and $(v,u)=1$.  We deduce that there exists a unique $k \pmod L$ with $w\equiv kv \pmod L$ where $k\equiv 1  \pmod{L/u}$ and $(k,u)=1$. Now let $g\equiv h^k \pmod p$ so that  $g^u\equiv h^{ku} \equiv h^u\equiv b \pmod p$ as 
$ku\equiv u \pmod L$ and
$a\equiv h^w\equiv h^{kv}\equiv g^v  \pmod p$.
\end{proof}

This proof can be modified  to show that if there exists  $h\pmod p$ with $\text{ord}_p(h)=L$ and an integer $w$
such that $a\equiv h^w  \pmod p$ and $b\equiv h^{u} \pmod p$ then $(u,w,L)$ is a semi-reduced triple which is equivalent to the reduced triple $(u,v,L)$ in Proposition \ref{Prop Generators}(b).

 We can immediately deduce how two-dimensional congruences can be used to describe pairs $(m,n), m,n\ge0$ for which $a^m-b^n$ is divisible by a fixed prime $p$:

 \begin{proposition} \label{Prop 1} Let $p$ be a prime that does not divide the integers $a$ and $b$.
There exists a unique reduced triple $(u,v,L)$ for which 
\[
  \{ (m,n)\in  (\mathbb Z_{\ge0})^2:\ a^m\equiv b^n \pmod p\} = S(u,v,L)\cap (\mathbb Z_{\ge0})^2.
  \]
Here $L=[r,s]$ where   $r=\text{ord}_p(a), s=\text{ord}_p(b)$  with  $u=\frac Ls$ and $(v,L)=\frac Lr$.
\end{proposition}  

\begin{proof} By Proposition \ref{Prop Generators}(b) there exists a unique residue $g\pmod p$ with $\text{ord}_p(g)=L$  such that $a\equiv g^{v} \pmod p$ and $b\equiv g^{u} \pmod p$ where $(u,v,L)$ is a    reduced triple for some unique integer $v$ (and therefore the reduced triple is unique).  Therefore $a^m\equiv b^n \pmod p$ if and only if $g^{mv}\equiv g^{nu} \pmod p$, which holds if and only if
$mv\equiv nu \pmod L$, as $\text{ord}_p(g)=L$.  The result follows.
\end{proof}

  \subsection{Primes yielding a given semi-reduced triple}
For a given triple $(p,a,b)$ where $p$ is prime and $a, b$ are positive integers not divisible by $p$, let
$$
(u_p(a,b), v_p(a,b), L_p(a,b))
$$ 
be the reduced triple given by Proposition~\ref{Prop 1}. 
For any given reduced triple $(u,v,L)$ we let 
\[
\mathcal P(u,v,L):=\bigg\{ (p,a,b) :\ (u_p(a,b), v_p(a,b), L_p(a,b))= (u,v,L)\bigg\}.
\]
We immediately deduce the following classification from Proposition~\ref{Prop 1} and Proposition \ref{Prop Generators}(b):

  \begin{lemma} \label{lem: create} Suppose that $(u,v,L)$ is a reduced triple. Then
 $(p,a,b)\in \mathcal P(u,v,L)$ if and only if $p\equiv 1 \pmod L$ and there exists a residue $g \pmod p$ of order $L$
 with  $a\equiv g^v \pmod p$ and $b\equiv g^u \pmod p$.
\end{lemma}

\subsection{Two-dimensional covering systems}
 A set of  integer triples $\{ (u_i,v_i,r_i): i=1,\dots,k\}$ is a  two-dimensional covering system if
 \[
 \bigcup_{i=1}^k S(u_i,v_i,r_i) = \mathbb Z \times \mathbb Z.
 \]
The covering system is \emph{minimal} if no proper subset covers all of  $ \mathbb Z_{\geq 1}\times \mathbb Z_{\geq 1}$. A covering system is \emph{semi-reduced} if it is made up of semi-reduced triples   (which is equivalent to covering $\Z^2$ by the sublattices 
$\langle  (0,r_i/u_i), (u_i,v_i)\rangle_{\Z}, i=1,\dots,k$).
A covering system is \emph{reduced} if it is made up of reduced triples; note that since any triple is equivalent to a unique reduced triple, therefore any covering system is equivalent to a reduced covering system.

  Two dimensional covering systems have a rich history \cite{CM, JM, PS, Sc, Si}  and very recently \cite{CK} (which is formulated in terms of sublattices covering $\Z^2$)
 but not in our context.
 
Our main tool is given by the following result:
     
 \begin{corollary} Let $Q$ be an integer coprime to $ab$. Then
  \[
  \{ (m,n): (a^m-b^n,Q)>1\} \supset  \mathbb Z_{\geq 0}\times \mathbb Z_{\geq 0}
  \]
  if and only if  $\{ S(u_p(a,b), v_p(a,b), L_p(a,b)): \text{Prime } p \text{ divides } Q\}$ is a two-dimensional  covering system.
 Moreover $Q$ is minimal (in that no proper divisor has a common factor with $a^m-b^n$ for all $m,n\geq 1$) if and only if the covering system is minimal.
 \end{corollary}

\begin{proof}
Now
 \[
 \{ (m,n): (a^m-b^n,Q)>1\}  = \bigcup_{p|Q} \{ (m,n): a^m\equiv b^n \pmod {p}\}.
 \]
(It is convenient, here and throughout, to define $(a^m-b^n,Q)$, when $m$ and $n$ might be negative integers, to equal the greatest common divisor of $Q$ and the numerator of $a^m-b^n$.)

Proposition \ref{Prop 1} implies that there is a reduced integer triple $(u_p(a,b), v_p(a,b), L_p(a,b))$
 for which 
 \[
  \{ (m,n): a^m\equiv b^n \pmod {p_i}\} = S(u_p(a,b), v_p(a,b), L_p(a,b)).
  \]
The result follows.
\end{proof}

\subsection{Constructing $a,b$ and $Q$ for a given two-dimensional covering system}
   
   On the other hand, suppose that we are given a reduced two-dimensional covering system $\{ (u_i,v_i,L_i): i=1,\dots,k\}$. We can construct
   distinct primes, $p_1,\dots,p_k$ and residues $a_i, b_i \pmod {p_i}$ with each $(p_i,a_i,b_i)\in \mathcal P (u_i,v_i,L_i)$ using Lemma  \ref{lem: create}; that is, select any prime $p_i\equiv 1 \pmod{L_i}$, any $g_i  \pmod{p_i}$ of order $L_i$, and then let
$a_i\equiv g_i^{v_i} \pmod {p_i}$ and $b_i\equiv g_i^{u_i} \pmod {p_i}$. Now
 let $Q=p_1\cdots p_k$ and determine $a,b \pmod Q$ for which  $a\equiv a_i \pmod {p_i}$ and  $b\equiv b_i \pmod {p_i}$ for each $i$, using the Chinese Remainder Theorem. Therefore
 \begin{align*}
 \{ (m,n): (a^m-b^n,Q)>1\}  &= \bigcup_{i=1}^k \{ (m,n): a^m\equiv b^n \pmod {p_i}\}  \\ 
&= \bigcup_{i=1}^k \{ (m,n): a_i^m\equiv b_i^n \pmod {p_i}\}  \\ 
&= \bigcup_{i=1}^k S(u_i,v_i,L_i)=\Z \times \Z \supset \Z_{\geq 0} \times \Z_{\geq 0} .
  \end{align*}

Given $(u_i,v_i,L_i)$ and a choice of prime $p_i$ there are $\phi(L_i)$ choices of the $g_i$ and then the $a_i$ and $b_i$ are determined, and so $a \pmod Q$ and $b\pmod Q$.  There are therefore  in total
   $\prod_i \phi(L_i)$ pairs of residues classes $a,b \pmod Q$ where  the $p_i$ are distinct primes $\equiv 1 \pmod {L_i}$.

\begin{example} The  two-dimensional reduced covering system $\{ (1,2,2), (2,1,2), (1,1,2)\}$ gives the following:
If $p,q,r$ are distinct odd primes for which 
 \[
 p|(a+1,b-1),\ q|(a-1,b+1),\ r|(a+1,b+1)
 \]
then $(a^m-b^n,pqr)>1$ for all integers $m,n\geq 0$.
\end{example}

\begin{example} The reduced two-dimensional covering system  $\{ (1,3,3), (3,1,3), (1,1,3), (1,2,3)\}$ yields, for distinct  primes $p_1,p_2,p_3,p_4$ all $>3$, if
\[
p_1|(a-1,b^2+b+1), p_2|(b-1,a^2+a+1), p_3|(b-a,a^2+a+1),  p_4|(b-a^2,a^2+a+1)
\]
then $(a^m-b^n,p_1p_2p_3p_4)>1$ for all integers $m,n\geq 0$.\smallskip

For instance if $p_1=7, p_2=13, p_3=19, p_4=31$ and $a, b$ satisfy:
$$\begin{cases}a\equiv 1\bmod 7 & \text{so that }7\mid a-1\\
a\equiv 3\bmod 13& \text{so that }13\mid a^2-a+1\\
a\equiv 7\bmod 19& \text{so that }19\mid a^2+a+1\\
a\equiv 5\bmod 31 & \text{so that }31\mid a^2+a+1,\\
\end{cases}
\quad
\begin{cases}b\equiv 2\bmod 7 & \text{so that }7\mid b^2-b+1\\
b\equiv 1\bmod 13& \text{so that }13\mid b-1\\
b\equiv 7\bmod 19& \text{so that }19\mid b-a\\
b\equiv 25\bmod 31 & \text{so that }31\mid b-a^2,\\
\end{cases}$$
we obtain:

\centerline{$(15226^m-67419^n,7\cdot 13\cdot 19\cdot 31)>1$ for all $m,n\geq 0$.}

\noindent We chose $b=67419$ rather then $b=13820$ (which both solve the system of congruences above) to avoid an even value of $b$.  
\end{example}

\begin{example} The \emph{trivial} reduced covering system $\{(1,1,1)\}$
corresponds to primes $p\mid (a-1,b-1)$ so that $p\mid a^m-b^n$ for all $m,n\ge0$.
Therefore, otherwise, we can restrict attention to pairs of integers $(a,b)$ for which  $(a-1,b-1)=1$, as well as $(a,b)=1$ (from section \ref{sec: Rmks}, remarks (i) and (ii)), so  one of $a$ and $b$ is odd, the other even.
\end{example}

\subsection{Properties of non trivial reduced covering systems}
All non-trivial semi-reduced covering systems take the form  $\{ (u_i,v_i,L_i): i=1,\dots,k\}$ with each $L_i>1$.
Since each $p_i\equiv 1 \pmod {L_i}$ according to the construction of Lemma \ref{lem: create}, we must have each $p_i\geq 3$ and so  $Q$ must be odd.

For $m=1,n=L_1\cdots L_k$ we must have some $i$ with $u_i=L_i$, and so $a\equiv 1 \pmod {p_i}$. Hence
$(a-1,Q)>1$   and $(b-1,Q)>1$ analogously, and so $a-1$ and $b-1$ must both have odd prime factors.

This last deduction implies that there is no covering system if one of $a$ or $b$ is either $2$, or is of the form $2^\ell+1$ for some $\ell\geq 1$. These   remarks lead us to make the following conjecture:

\subsection*{Prime values in the exceptional cases}
 \emph{Let $a=2$ or $a=2^\ell+1$ for some integer $\ell\geq 1$, and suppose that
 $b$ is a positive integer which is coprime with $a$ and of opposite parity to $a$. Then  there are infinitely many primes of the form $|b^n-a^m|$ as $m$ and $n$ vary over the  integers $\geq 1$.}

\section{Heuristic for the number of primes $a^m-b^n$}

We first estimate $\#\{ (m,n): |a^m-b^n|\leq x\}$ and use the Cram\'er heuristic to guess at the number of prime values. We then
adjust this with local probabilities to guess at the number of primes.

\subsection{Linear forms in logarithms}
.  We will use Baker's theorem \cite{Bak} in the following form to show that the above has finitely many solutions: There exists a number $C=C(a,b)$ such that if $m,n\geq 1$ then
\[
|m\log a-n\log b|\geq (m+n)^{-C}
 \]
 Taking exponentials of both sides and multiplying through by $b^n$ we get
 \begin{equation} \label{eq: LB}
    |a^m-b^n| \gg b^n(m+n)^{-C}
 \end{equation}
 and this goes to $\infty$ as $a^m,b^n\to \infty$. 
 Thus  $|a^m - b^n |$ equals any given value for only finitely many positive pairs $(m,n)$.

\subsection{The number up to $x$}
We wish to count the number of $m,n\geq 1$ with $0<a^m-b^n\leq x$. 
If $m\leq \frac{\log x}{\log a}$  and $n\leq \frac{\log a^m}{\log b}$ then $b^n<a^m\leq x$, and the number of such pairs is
\[
\sum_{m\leq \frac{\log x}{\log a}} \bigg\lfloor \frac{\log a^m}{\log b} \bigg\rfloor = 
\sum_{m\leq \frac{\log x}{\log a}} \bigg( m \frac{\log a}{\log b}+O(1) \bigg) = \frac{(\log x)^2}{2\log a \log b} +O\bigg( \frac{\log x}{\log a}+\frac{\log x}{\log b} \bigg) .
\]
Now we consider $m> \frac{\log x}{\log a}$ and let $n_m$ be maximal with $b^{n_m}<a^m$.
  Equation \eqref{eq: LB} yields that 
  \[  x\geq a^m-b^n \gg a^mm^{-C} \]
  and so $m\leq \frac{\log x}{\log a}+O( C \log\log x)$ and therefore there are $\ll C \log\log x$ such pairs   
  $(m,n_m)$. If $n\leq n_m-1$ and 
\[
x\geq a^m-b^n\geq a^m-b^{n_m}/b>a^m(1-1/b),
\]
so $a^m\leq \frac b{b-1} x$ and therefore $m\leq \frac{\log x+O(1/b)}{\log a}$. This yields at most one   value of $m$, and so the number of such pairs $(m,n)$ is $\ll \frac{\log x}{\log b}$ (which can be attained if $a^m=x+1$).

Adding these estimates, together with the analogous argument for when $0<b^n-a^m\leq x$ we obtain
\[
\#\{ m,n\geq 1: |a^m-b^n|\leq x\} = \frac{(\log x)^2}{\log a \log b} +O\bigg( \frac{\log x}{\log a}+\frac{\log x}{\log b} \bigg) .
\]

\subsection{The Cram\'er heuristic} A randomly chosen integer near $x$ is prime with probability about $\frac 1{\log x}$. If a set of integers is chosen more or less randomly then a first guess at the number of primes in the set can be obtained by applying this heuristic.

Now for each $n\leq n_m$ we have the ``probability'' that $a^m-b^n$ is prime is about
\[
\frac 1{\log (a^m-b^n)} = \frac 1{\log (a^mm^{O(1)})} \sim  \frac 1{\log a^m}
\]
while the number of such $n$ is $\frac{\log a^m}{\log b}+O(1)$. So the``expected'  number of primes in
$\{ a^m-b^n: 1\leq n\leq \frac{\log a^m}{\log b}\}$ is $\sim \frac{\log a^m}{\log b}\cdot  \frac 1{\log a^m}\sim \frac 1{\log b}$.
Summing this up over all $m$ with $a^m\leq x$ (or $a^m\leq \frac b{b-1} x$) we expect $\sim \frac{\log x}{\log a \log b}$ primes. 

Combining what we get here with the analogous argument for $a^m<b^n$ we obtain the guess:
\[
\#\{ m,n\geq 1: |a^m-b^n| \text{ is prime and } a^m, b^n\leq x\} \sim  \frac{2\log x}{\log a \log b}.
\]
We will need to adjust the constant to take account of divisibility by small primes, but still we believe that if $x=a^y$ then the number of primes should be linear in $y$. We test this next.

\subsection{Linearity}

Suppose that $1<a<b$ and that there is no covering system for $a^m-b^n$. Let
\[
\pi_{a,b}(y):=\# \{  m,n\geq 1: |a^m-b^n| \text{ is prime and } a^m, b^n\leq a^y\} .
\]
Since $a^m-b^n$ is coprime to $ab$ we can most simply adjust the above guess by multiplying through by $\frac a{\phi(a)} \frac b{\phi(b)} $ to obtain the new guess
\[
\#\{ m,n\geq 1: |a^m-b^n| \text{ is prime and } a^m, b^n\leq x\} \sim  \frac{2ab \log x}{\phi(ab) \log a \log b}.
\]

We define $N_k=\pi_{a,b}(100k)-\pi_{a,b}(100(k-1))$ for $k\geq 1$. Our heuristic suggests  that these numbers should each be roughly 
\[
G_1(a,b):= \frac{200ab }{\phi(ab) \log b}
\]
 our ``first guess'' for the $N_k$-values, which we test in the next table:
\begin{table}[ht]
\begin{center}
\begin{tabular}{|  c | c c c c c c c | c |}
\hline\newline
 $a,b$  & $k=$1 & 2 & 3 & 4 & 5 & 6 & 7 & $G_1(a,b)$\\
\hline\newline
2, 3 & 417 & 411 & 459 & 433 & 409 & 438 & 446 & 546 \\
3, 4 & 294 & 299 & 284 & 297 & 290 & 283 & 263&  433 \\
2, 5 & 249 & 271 & 244 & 234 & 244 & 245 & 275 & 311\\
4, 5 & 293 & 245 & 278 & 290 & 253 & 294 & 269 & 311\\
5, 6 & 271 & 282 & 253 & 283 & 306 & 282 & 261 & 419\\
2, 7& 175 & 171 & 185 & 148 & 202 & 172 & 160& 240\\
\hline 
\end{tabular}
    \end{center} \smallskip
    \caption{$N_k$-values for various pairs $(a,b)$ with $b>a>1$, and our ``first guess'', $G_1(a,b)$.}
    \end{table}

\noindent The data on each row seems roughly constant, and so persuades us that the $\pi_{a,b}(y)$ are indeed approximately linear.
However our ``first guess'' is consistently too large, so
we next try to adjust our guess to get a more accurate fit with the data.


\subsection{Local adjustments, and special form adjustment}
It is usual, when  guesstimating the number of primes in a given set of integers, to adjust one's guesstimate depending on how the set of integers is distributed mod $p$ for each prime $p$. If the set is the set of values of a polynomial, then the distributions mod $p$ are independent for different $p$, since the values are periodic mod $p$ with period $p$.
However this is not so in our case. For example we see that $3$ divides $2^n-1$ if and only if $n\equiv 0 \pmod 2$ and 
$5$ divides $2^n-1$ if and only if $n\equiv 0 \pmod 4$, so that $15$ divides $2^n-1$ if and only if $n\equiv 0 \pmod 4$. 

Instead of working mod $p$, we need to work modulo a sequence of composite integers $Q_1,Q_2,\dots$ such that every prime $p$ divides  $Q_j$ for all $j\geq j_p$.  One idea is to have  $Q_j$ be the product of the $j$ smallest primes (which is essentially the usual choice), but  we saw in \cite{GG} that other choices might be more natural.

Potentially there is a second complicating issue. If $g:=(m,n)>1$ then $a^{m/g}-b^{n/g}$ divides $a^m-b^n$ and so $a^m-b^n$ is composite unless
$a^{m/g}-b^{n/g}=\pm 1$. Mihailescu's theorem \cite{Mi} (which was Catalan's conjecture) implies that either
$\{ a^{m/g}, b^{n/g}\}=\{ 3^2,2^3\}$ or $m/g=1$ or $n/g=1$.
If $n=g$ then $n$ divides $m$, say $m=nk$ where $b=a^k\mp 1$;
if $m=g$ then $m$ divides $n$, say $n=m\ell$ where $a=b^\ell\pm 1$.
However in all three cases this plays a role for $O(N)$  $(m,n)$-pairs with $m,n\leq N$, whereas there are $\asymp N^2$ pairs, so these cases effect a vanishing proportion of pairs $(m,n)$, so can be ignored.

The prime factors of $ab$ never divide $a^m-b^n$ with $m,n\geq 1$ but all other primes do for some $m,n$ values (since $p|a^{p-1}-b^{p-1}$).
Let $P_i$ be the $i$th smallest prime power,
\[
r_k=[P_1,\cdots,P_k] \text{ and } q_k=(a^{r_k}-1, b^{r_k}-1).
\]
The proportion of $a^m-b^n$ values that have no factor in common with $q_k$ is given by
\[
\frac 1{x^2} \#\{ m,n\leq x: (a^m-b^n,q_k)=1\} \sim \frac 1{r_k^2} \#\{ m,n\leq r_k: (a^m-b^n,q_k)=1\} .
\]
We wish to incorporate the criterion $(m,n)=1$.   This fits best in our approach if we work instead with $(m,n,r_k)=1$ since that will find all common factors of any given $m$ and $n$ once $k$ is sufficiently large, and the criteria is now also  $r_k$-periodic. Therefore
the proportion of $a^m-b^n$ values that have no factor in common with $q_k$ and with $(m,n)$ coprime with $r_k$   is given by
\[
\frac 1{x^2} \#\bigg\{ m,n\leq x: {(a^m-b^n,q_k)=1\above 0pt \& \  (m,n,r_k)=1}\bigg\} \sim
\frac 1{r_k^2} \#\bigg\{ m,n\leq r_k: {(a^m-b^n,q_k)=1\above 0pt \& \  (m,n,r_k)=1}\bigg\} 
\]
As $k\to \infty$ this incorporates divisibility by small primes as well as $(m,n)=1$, as desired.
For regular integers, the analogous probability is $\phi(q_k)/q_k$, and so the adjustment to the Cram\'er heuristic, taking into account divisibility by the small primes, is 
\[
 \kappa_{a,b}(k) :=\frac {q_k}{\phi(q_k)} \cdot \frac 1{r_k^2} \#\bigg\{ m,n\leq r_k: {(a^m-b^n,q_k)=1\above 0pt \& \  (m,n,r_k)=1}\bigg\} \]
 hopefully becoming more accurate as $k$ increases.
Indeed we believe that 
\[
\lim_{k\to \infty} \kappa_{a,b}(k) \text{ exists, and equals a non-zero constant } \kappa_{a,b},
\]
and therefore we conjecture that 
 \[
 \boxed{ \#\{ |a^m-b^n|\leq x: |a^m-b^n| \text{ is prime}\} \sim  \frac{2ab\, \kappa_{a,b}}{\phi(ab)\log a\log b}\cdot \log x;}
 \]
and so, in the introduction, we have  $c_{a,b}:= \frac{2ab\, \kappa_{a,b}}{\phi(ab)\log a\log b}$.

\subsection{Convergence of the constant}
In practice we cannot easily calculate past $r_7=2520$, and we have no idea how to prove that the $\kappa_{a,b}(k)$ converge   to a positive value. Here is some data of $\kappa_{a,b}(k)$-values:

\begin{table}[ht]
\begin{center}
\begin{tabular}{|  c | c c c c c c  | }
\hline\newline
 $a,b$  & $k=$2& 3 & 4 & 5 & 6 &  7 \\
\hline\newline
2, 3 &  .713 & .746 & .747 & .740 & .749 & .777   \\
3, 4 &   .702 & .746 & .665 & .683 & .692 & .705  \\
2, 5 &  .681 & .737 & .709 & .696 & .699 & .725  \\
4, 5 & .778 & .843 & .739 & .715 & .718 & .721   \\
5, 6 &  .670 & .689 & .698 & .679 & .678 & .666     \\
2, 7&  .667 & .621 & .636 & .659 & .650 & .705  \\
\hline 
\end{tabular}
    \end{center} \smallskip
     \caption{$\kappa_{a,b}(k)$-values for increasing $k$, appear to converge}
    \end{table}

\noindent    The entries in the rows of Table 2 are not varying too much and perhaps converging,   but it is  rather scant evidence.

We developed the $\kappa_{a,b}(k)$ sequence to test our conjecture for prime values of $a^m-b^n$. Therefore it is best to now 
 compare the mean value of the $N_k$'s from the first table with our ``second guess'',
\[
G_2(a,b):=\frac{200ab\, \kappa_{a,b}(7)}{\phi(ab) \log b} 
\]

 \begin{table}[ht]
\begin{center}
\begin{tabular}{|  c | c | c   | }
\hline\newline
 $a,b$  & Mean $\#$ of primes &  $G_2(a,b)$ \\
\hline\newline
 2, 3 &  430 & 424   \\
3, 4 &   287  & 305  \\
2, 5 &  252 & 225  \\
4, 5 & 275 & 224   \\
5, 6 &  277 &  279     \\
2, 7&  173 & 169  \\
\hline 
\end{tabular}
    \end{center} \smallskip
      \caption{Mean of $N_k$-values for various   $b>a>1$, and our ``second guess'', $G_2(a,b)$.}
    \end{table}                           
 This looks fairly persuasive, but it would be good to collect more evidence.                    
   
\section{Calculations reveal}  
We have claimed that if there is no compelling reason why not, then there are lots of prime values of $|a^m-b^n|$.
In this section we discuss calculations designed to determine how well our predictions are reflected in actual data.
Given the fast growth of $a^m$ and $b^n$ as functions of $m$ and $n$, we can only compute a fairly limited amount of data for each pair $a$ and $b$, but nonetheless, what we can compute has given us confirmation that our predictions seem right.

We restrict our attention to pairs $(a,b)$ for which it is feasible, at first sight, that there are lots of prime values of $|a^m-b^n|$.
Let $\mathcal N$ be the set of integer pairs $(a,b)$ where $b>a\geq 2$ with $(a,b)=(a-1,b-1)=1$ and $a$ are $b$ are not both perfect $p$th powers for some prime $p$ (in which case if $a=A^p, b=B^p$ then $a^m-b^n$ is divisible by $A^m-B^n$). We have seen that if 
$(a,b)\not\in \mathcal N$ then it is easy to show that $a^m-b^n$ takes only finitely many prime values. We have not excluded all $(a,b)$-pairs that correspond to a covering system, only the simplest that are easy to identify immediately. We might expect other pairs with more complicated covering systems will ``self-identify'' by yielding few prime values.

Given integers $(a,b)\in \mathcal N$ we select $x_{a,b}$ so that 
\[
\frac{ab \log x_{a,b}}{\phi(ab) \log a \log b}=100.
\] 
We will calculate $\#\{ \text{Primes }  |a^m-b^n|: a^m,b^n\leq x_{a,b}\}$; our prediction suggests that this is
$ \approx  200 \kappa_{a,b}$ primes.  Now let $M_{a,b}:=\frac{\phi(ab)}{ab}  \log b, N_{a,b}:=\frac{\phi(ab)}{ab}  \log a$ so that 
$a^m,b^n\leq x_{a,b}$ if and only if $m  \leq \ 100 M_{a,b}$ and $n\leq 100N_{a,b}$, and therefore
we will calculate
\[
\Pi_{a,b}(100):=\# \{ m\leq 100 M_{a,b}, n\leq 100N_{a,b}: |a^m-b^n| \text{ is prime }\}.
\]

We took the viewpoint that if $\Pi_{a,b}(100)\geq 25$ then there are probably infinitely many primes of the form
$|a^m-b^n|$, and we should investigate further when $\Pi_{a,b}(100)< 25$. In fact there are only five such pairs with $b\leq 100$:

\begin{table}[ht]
\begin{center}
\begin{tabular}{| c| c |  c   | }
\hline\newline
 $a$ & $b$  &   $\Pi_{a,b}(100)$ \\
\hline\newline
 9 & 74 &  20   \\
29 & 34 &   1  \\
34 & 41 &  1  \\
51 & 64 & 1   \\
59 & 86   & 0 \\
\hline 
\end{tabular}
    \end{center} \smallskip
      \caption{Pairs $(a,b)\in \mathcal N$ with $b\leq 100$ and $\Pi_{a,b}(100)< 25$.}
    \end{table}                           

In the first example we found $\Pi_{9,74}(100)=20, \Pi_{9,74}(200)=43, \Pi_{9,74}(300)=62,\dots$ which looks very much like linear growth, going to infinity. We found no additional primes in the other four examples and then looked for  covering systems. For three of them we found
\[ 7|(29-1, 34+1), 3|(29+1, 34-1),5|(29+1, 34+1) \]
\[ 3 |(34-1, 41+1),  5 |(34+1, 41-1),  7 |(34+1, 41+1)\]
\[  29 |(59-1, 86+1), 5 |(59+1, 86-1), 3|(59+1, 86+1)\]
which are each as in example 1. 

 The case $a=51, b=64$ is exactly the case discussed in Section \ref{sec: powers2}.
 
To compute further we define $ \mathcal N^+$ to exclude more possible covering systems:
 Let $ \mathcal N^+$ be the set of integer pairs $(a,b)$ where $b>a\geq 2$ with $(a,b)=(a-1,b-1)=1$ for which neither $a$ nor $b$ are perfect powers, and at least one of $(a-1,b+1)=1, (a+1,b-1)=1, (a+1,b+1)=1$ holds (else we have a covering system). The only exceptional cases with $a+b\leq 500$ and $(a,b)\in \mathcal N^+$ are

\begin{table}[ht]
\begin{center}
\begin{tabular}{| c| c |  c   | }
\hline\newline
 $a$ & $b$  &   $\Pi_{a,b}(100)$ \\
\hline\newline
26 &  149 & 19 \\
68 &  133 & 21 \\   
67 &  186 & 21 \\   
13 &302 &0 \\   
37 &284 &24 \\   
22 &321 &24 \\   
13 &356 & 0 \\
128 & 253 & 12 \\   
43 & 342 & 23 \\   
122 & 307 & 0 \\   
191 & 254 & 20 \\  
 202 & 251 & 20 \\   
 161 & 304 & 5 \\   
 146 & 323 & 23\\
  \hline 
\end{tabular}
    \end{center} \smallskip
      \caption{Pairs $(a,b)\in \mathcal N^+$ with $a+b\leq 500$ and $\Pi_{a,b}(100)< 25$.}
    \end{table}               
    
In the three cases where we get no primes we found covering systems:\\
Both $(a,b)=(13, 302)$ and $(a,b)=(122,307)$ correspond to the semi-reduced covering system: \\
$\{(1,2,2),(2,1,2),(1,1,4),(1,3,4)\}$ so that
 \[
 13^m-302^n\equiv
 \begin{cases}
   0 \pmod 7 &\text{  if } m\equiv 0 \pmod 2;\\
 0 \pmod {3} &\text{  if } n\equiv 0\pmod 2;\\
 0 \pmod {17} &\text{  if } m\equiv n\pmod 4;\\
  0 \pmod {5} &\text{  if } m\equiv -n\pmod 4,
 \end{cases}
 \]
 and
 \[
 122^m-307^n\equiv
 \begin{cases}
   0 \pmod 3 &\text{  if } m\equiv 0 \pmod 2;\\
 0 \pmod {11} &\text{  if } n\equiv 0\pmod 2;\\
 0 \pmod {5} &\text{  if } m\equiv n\pmod 4;\\
  0 \pmod {13} &\text{  if } m\equiv -n\pmod 4;
 \end{cases}
 \]
  while  $(a,b)=(13,356)$ corresponds
to the semi-reduced covering system: \\
$\{(1,2,2),(1,1,2),(4,1,4),(2,1,4)\}$  so that
 \[
 13^m-356^n\equiv
 \begin{cases}
   0 \pmod 3 &\text{  if } n\equiv 0 \pmod 2;\\
 0 \pmod {7} &\text{  if } m\equiv n\pmod 2;\\
 0 \pmod {5} &\text{  if } m\equiv 0\pmod 4;\\
  0 \pmod {17} &\text{  if } m\equiv 2 \pmod 4 \text{ and } n \equiv 1\mod 2.
 \end{cases}
 \]
 In the case $161^m-304^n$ we computed further and found more primes, so we believe that in all the other examples we would get infinitely many primes. To test our quantitative conjecture we determined how many primes there are up to the point that our prediction claimed there would be 100 primes. Thus we calculated
 $P_{a,b}(100):=\#\{ \text{Primes }  |a^m-b^n|: a^m,b^n\leq X_{a,b}\}$ where $X_{a,b}$ is selected so that
 \[
  \frac{2ab\kappa_{a,b}(7)}{\phi(ab)\log a\log b}\cdot \log X_{a,b}=100,
 \]
 \begin{table}[ht]
\begin{center}
\begin{tabular}{|  c | c | c c  | }
\hline\newline
 $a,b$  &  $P_{a,b}(100)$  & $\kappa_{a,b}(7)$ & Prediction \\
\hline\newline
26, 149 &  115  & .085    &  100   \\
68,  133 &  86 & .134    &   100  \\
67,  186 &  86 &  .095   &   100   \\
 37, 284 &  94 &  .188   &   100   \\
 22, 321 & 189 & .164 &   100   \\
128, 253 &  102    & .070 & 100\\
 43, 342 & 77 & .163 & 100\\
191, 254 &   99   & .151 & 100\\
202, 251 &  104    & .082 & 100\\
161, 304 &  57    & .091 & 100\\
146, 323 &110 & .180 & 100\\
\hline 
\end{tabular}
    \end{center} \smallskip   
       \caption{$P_{a,b}(100)$ when primes of the form $|a^m-b^n|$ seem sparse}
    \end{table}  
    
The data mostly corresponds well with the prediction, though it would be good to better understand the outliers here.


 \bibliographystyle{plain}

\begin{thebibliography}{99}

\bibitem{Bak} Alan Baker,
\emph{The theory of linear forms in logarithms}, in
``Transcendence theory: advances and applications'', Boston, MA: Academic Press, 
1977, pp. 1--27

\bibitem{CM} Todd Cochrane and Gerry Myerson,
\emph{Covering congruences in higher dimensions},
Rocky Mountain J. Math. \textbf{26} (1996), 77--81.

\bibitem{CK} J. E. Cremona, P. Koymans,
\emph{Lattice coverings and homogeneous covering congruences},
arxiv:2601.03212

\bibitem{JM}  Boping Jin  and Gerry Myerson,
\emph{Homogeneous covering congruences and subgroup covers},
J. Number Theory \textbf{110} (2005), 120--135.

\bibitem{GG} Jon Grantham and Andrew Granville,
\emph{Fibonacci primes, primes of the form $2^n-1$  and beyond},
J. Number Theory \textbf{261} (2024), 190--219.

\bibitem{Mi} Preda Mih\u ailescu,  
\emph{Primary cyclotomic units and a proof of Catalan's conjecture},
J. Reine Angew. Math. \textbf{572} (2004), 167--195.


\bibitem{PS} \u S. Porubsk\' y and  J. Sch\" onheim,
\emph{Covering systems of Paul Erd\" os. Past, present and future} in ``Paul Erd\" os and his mathematics, I'',
 (Budapest, 1999), 581--627.
Bolyai Soc. Math. Stud., 11
J\'anos Bolyai Mathematical Society, Budapest, 2002

\bibitem{Sc} Andrzej Schinzel, 
\emph{On homogeneous covering congruences},
Rocky Mountain J. Math. \textbf{27} (1997), 335--342. 

\bibitem{Sg} Carl L. Siegel,
\emph{The integer solutions of the equation $y^2=ax^n+bx^{n-1}+\cdots+k$},
 J. London Math. Soc. \textbf{1} (1926), 66--68. (Published under the pseudonym ``X''.)

\bibitem{Sg2} Carl L. Siegel,
\emph{Die Gleichung  $ax^n-by^n=c$},
 Math. Ann. \textbf{114} (1937), 57--68

\bibitem{Si} R. J. Simpson,
\emph{Covering systems of homogeneous congruences},
Rocky Mountain J. Math. \textbf{28} (1998), 1125--1133.



\end{thebibliography}

\enddocument